\pdfoutput=1 %may be required for arXiv
\documentclass[11pt]{amsart}
\usepackage{epsfig,overpic,comment}%graphics,amssymb
\usepackage[usenames,dvipsnames,svgnames,table]{xcolor}
\usepackage[pagebackref,linktocpage=true,colorlinks=true,linkcolor=Blue,citecolor=BrickRed,urlcolor=RoyalBlue]{hyperref}
\usepackage[alphabetic,msc-links,abbrev]{amsrefs} %shortalphabetic,msc-links,backrefs

\newtheorem{thm}{Theorem}[section]
\newtheorem{lem}[thm]{Lemma}
\newtheorem{cor}[thm]{Corollary}
\newtheorem{prop}[thm]{Proposition}

\theoremstyle{definition}

\newtheorem{note}[thm]{Note}

\newcommand{\R}{\mathbf{R}}
\newcommand{\RP}{\mathbf{RP}}
\newcommand{\ol}{\overline}
\newcommand{\C}{\mathcal{C}}

\renewcommand{\d}{\partial}

\renewcommand{\S}{\mathbf{S}}

\renewcommand{\tilde}{\widetilde}

\DeclareMathOperator{\dist}{dist}
\DeclareMathOperator{\inte}{int}
\DeclareMathOperator{\conv}{conv}

\DeclareMathOperator{\Length}{Length}
\DeclareMathOperator{\II}{II}
\DeclareMathOperator{\vol}{vol}

\title[Nonnegatively curved hypersurfaces]{Nonnegatively curved hypersurfaces with\\ free boundary on a sphere}

\author{Mohammad Ghomi}
\address{School of Mathematics, Georgia Institute of Technology,
Atlanta, GA 30332, USA}
\email{ghomi@math.gatech.edu}
\urladdr{www.math.gatech.edu/$\sim$ghomi}

\author{Changwei Xiong}
\address{Mathematical Sciences Institute, Australian National University, Canberra ACT 2601, Australia}
\email{changwei.xiong@anu.edu.au}
\urladdr{maths.anu.edu.au/people/changwei-xiong}

\date{\today \,(Last Typeset)}
\thanks{The research of M.G. was supported in part by NSF grant DMS-1711400. The research of C.X. was supported by the ARC Laureate Fellowship FL150100126.}

\begin{document}

\begin{abstract}
We prove that in Euclidean space $\R^{n+1}$ any compact immersed nonnegatively curved hypersurface  $M$ with free boundary on the  sphere $\S^n$  is an embedded convex  topological disk. In particular, when the $m^{th}$ mean curvature of $M$ is constant,  for any $1\leq m\leq n$, $M$ is a spherical cap or an equatorial disk.
\end{abstract}

\maketitle

\section{Introduction}\label{sec1}

A fundamental  result in submanifold geometry is the convexity of closed hypersurfaces with nonnegative (sectional) curvature immersed in Euclidean space $\R^{n+1}$.  Hadamard \cite{Had97} observed this phenomenon for $n=2$ and positive curvature in 1897. Chern and Lashof \cite{chern&lashof:tight2} extended Hadamard's theorem to nonnegative curvature,  and Sacksteder \cite{sacksteder:convex}  proved the full result in all dimensions, by reducing it to a nonsmooth analogue due to van Heijenoort \cite{vH:convex}. Similar phenomena have also been established   in the sphere $\S^{n+1}$ and the hyperbolic space $\mathbf{H}^{n+1}$, by do~Carmo and Warner \cite{docarmo&warner}, and  Currier \cite{Cur89} respectively.
We  add a  result to this genre for surfaces with boundary:

\begin{thm}\label{thm:main}
Any compact $\C^\infty$ immersed nonnegatively curved hypersurface  $M$ in $\R^{n+1}$ with free boundary $\d M$ on  $\S^n$ is an embedded convex topological disk.
\end{thm}

\emph{Free boundary}  here means that $M$ is orthogonal to $\S^n$ along $\d M$---a condition which arises naturally  in variational problems, e.g., see \cite{cgr:total, cgr:relative}.  Furthermore, by \emph{convex} we mean that the surface lies on the boundary of a convex body (see Section \ref{subsec:terminology} for basic definitions). Surfaces  with free boundary  have received  much attention recently, especially since Fraser and Schoen \cite{FS11} studied Steklov eigenvalues  of minimal submanifolds in the ball $B^{n+1}$. These works often point to a strong similarity between closed hypersurfaces of $\S^{n+1}$ and hypersurfaces with free boundary in $B^{n+1}$. The above theorem is another instance of this phenomenon, and also yields the following  characterization for umbilical hypersurfaces, which mirrors results of Hartman \cite{Har78} and Cheng and Yau \cite{cheng&yau}:

\begin{cor}\label{cor:cmc}
Let $M$ be as in Theorem \ref{thm:main}. Suppose that the $m^{th}$ mean curvature of $M$ is constant for some $1\leq m\leq n$. Then $M$ is either a spherical cap or an equatorial disk.
\end{cor}

The proof of Theorem \ref{thm:main} employs the classical results mentioned above together with a host of more recent techniques \cite{ghomi:stconvex,alexander&ghomi:chp, ghomi:verticesC,ghomi&howard:tancones}. First we show that every component $\Gamma$ of $\d M$ is convex in $\S^n$ (Sections \ref{sec:dM}, \ref{sec:dMII}, \ref{sec:dMIII}). Next we glue a convex disk along each $\Gamma$ to extend $M$ to a closed $\C^1$ hypersurface $\ol M$, which is $\C^\infty$ and nonnegatively curved almost everywhere (Section \ref{sec:extension}). Finally we prove the convexity of $\ol M$ by adapting a proof of Sacksteder's theorem (in the compact case) due to do Carmo and Lima \cite{dCL69} (Section \ref{sec:dCL}). Proofs of the last two steps are the same in all dimensions; however, the first step involves much more work for $n=2$ (Sections  \ref{sec:dMII}, \ref{sec:dMIII}), which forms the bulk of this paper. Corollary \ref{cor:cmc} follows quickly from Theorem \ref{thm:main} via Alexandrov's reflection technique and the generalized Delaunay theorem for rotational surfaces (Section \ref{sec:cmc}). The following notes show that the conditions of Theorem \ref{thm:main} are sharp.

\begin{note}
The free boundary condition in Theorem \ref{thm:main} is essential. Consider for instance the surface $\Sigma\subset\R^3$ given by $z = x^3(1 + y^2)$ and $|y|<1/2$, which appears in Sacksteder \cite{sacksteder:convex}. This surface is nonnegatively curved,  but fails to be convex in any neighborhood of the origin. Let $\lambda\Sigma$ denote homothetic copies of $\Sigma$ for $\lambda>2$, and
$M_\lambda$ be the component of  $\lambda\Sigma$ contained in $\S^2$. As $\lambda\to \infty$, $M_\lambda$ becomes arbitrarily close to being orthogonal to $\S^2$, while it remains nonconvex.
\end{note}

\begin{note}\label{note:torus}
Theorem \ref{thm:main} may not hold if $\S^n$ is replaced by another convex surface. For instance let $C$ be the cylinder $x^2+y^2=1$ in $\R^3$, $T$ be the torus obtained by revolving the circle given by $(x-1)^2+z^2=1/4$ and $y=0$ around the $z$-axis, and  $M$ be the portion of $T$ outside $C$. Then $M$ is a nonnegatively curved surface with free boundary on $C$, which is not simply connected. We may regard $T$ as the image of a multiple covering by another torus, in which case $M$ will also fail to be embedded. Finally let $T'$ be the portion of $T$ contained in the region $\{x\leq 0\}\cup\{y\leq 0\}$ of $\R^3$,  fill in the boundary components of $T'$ with disks, and let $M'$ be the portion of the resulting surface which lies outside $C$. Smoothing the corners of $M'$ yields a nonnegatively curved surface with free boundary on $C$ which is not convex.
\end{note}

\begin{note}
The free boundary condition in Theorem \ref{thm:main} cannot be generalized to a constant angle (or capillary) condition along $\d M$. Indeed let $T$ be the torus of revolution in Note \ref{note:torus}, $S$ be the sphere of radius $\sqrt{5}/2$ centered at $o$, and $M$ be the portion of $T$ outside $S$. Then $M$ is a nonnegatively curved surface which meets $S$ at a constant angle along its boundary, but is not simply connected. As discussed in Note \ref{note:torus}, one may also construct  nonconvex versions of this example.
\end{note}

\begin{note}
The compactness requirement in Theorem \ref{thm:main} may not be weakened to metric completeness (in the sense of Cauchy): take any smooth closed curve $\gamma\colon\S^1\to\S^2$, which is not convex, and let  $M$ be generated by $\lambda\gamma(t)$ for $\lambda\geq 1$.
\end{note}

\section{Preliminaries: Local Convexity}\label{sec:prelim}
A chief difficulty in working with nonnegative (as opposed to strictly positive) curvature is the absence of local convexity. We deal with this issue  by slicing the surface with hyperplanes that separate its interior points from its boundary, and thus generate convex caps, as we review in this section. More extensive background may be found in \cite{alexander&ghomi:chp,ghomi:rosenberg}.

\subsection{Basic terminology}\label{subsec:terminology}
Throughout this work, $\R^{n+1}$ denotes $(n+1)$-dimensional Euclidean space with standard metric $\langle\cdot,\cdot\rangle$ and origin $o$. Furthermore $\S^n$, $B^{n+1}$ denote respectively the unit sphere, and the (closed) unit ball  in $\R^{n+1}$.
Unless stated otherwise, we will assume that $M$ is a compact connected $(n\geq 2)$-dimensional manifold, with (nonempty) boundary $\d M$. We say that $M$ is a (topological) \emph{disk} if it is homeomorphic to $B^n$. An \emph{equatorial disk} is the intersection of $B^{n+1}$ with a hyperplane through $o$.
We always assume that $M$ is topologically immersed in $\R^{n+1}$, i.e., there exists a continuous locally one-to-one map $f\colon M\to\R^{n+1}$. We say that $M$ is $\C^k$ if $f$ is $\C^k$, and a subset of $M$ is \emph{embedded} if $f$ is one-to-one on that set. To reduce notational clutter, we will suppress $f$, and identify $M$ locally with its image under $f$. As far as the proof of Theorem \ref{thm:main} is concerned, we may assume without loss of generality that $M$ is \emph{orientable}, after replacing it by its double  cover if necessary. So we will assume that $M$ is orientable.
A \emph{convex body} $K\subset\R^{n+1}$ is a compact convex set with interior points.
We say that  $M$ is \emph{locally convex} at a point $p$ if there exists an open neighborhood $U$ of $p$ in $M$ which lies on the boundary of a convex body $K\subset\R^{n+1}$.  We say that $M$ is locally convex if it is locally convex everywhere, and $M$ is \emph{convex} if it lies embedded on the boundary of a convex set  with interior points in $\R^{n+1}$. If $M$ is $\C^2$ and has nonnegative (sectional) curvature, we say that it is \emph{infinitesimally convex}. Note that every $\C^2$ locally convex hypersurface is necessarily infinitesimally convex, but the  converse in general is not true.

\subsection{Convex caps}
A \emph{convex cap} $C$ in $\R^{n+1}$ is a convex disk whose boundary lies on a hyperplane $H$, while the rest of it does not. We say that $C$ is \emph{spherical} if it lies on a round sphere.
In \cite{vH:convex} van Heijenoort employed convex caps to show that  a complete locally convex hypersurface immersed in $\R^{n+1}$ is convex provided that it is \emph{locally strictly convex} at one point $p$; see also \cite{jonker&norman}. The latter condition  means that there passes a hyperplane through $p$ which intersects an open neighborhood of $p$ in $M$ only at $p$. In particular note that local strict convexity does not necessarily imply that the curvature is positive (e.g. consider the surface $z=x^4+y^4$ in $\R^3$). Sacksteder \cite{sacksteder:convex}  showed that a complete nonnegatively curved $\C^{n+1}$ hypersurface $M$ immersed in $\R^{n+1}$ is locally convex provided that it has a point of positive curvature, which yields the convexity of $M$ via van Heijenoort's theorem.
The following observation is a quick consequence of these results via a projective transformation:

\begin{lem}\label{lem:clipping}
Let  $H$ be a hyperplane, $H^+$ be one of the closed half spaces of $H$, and  $M^+$ be a component of $M$ in  $\inte(H^+)$. Suppose that $M^+$ is disjoint from $\d M$. Furthermore suppose that either $M$ is locally convex, or else is $\C^{n+1}$ and infinitesimally convex. Then the closure of $M^+$ is  a convex cap.
\end{lem}
\begin{proof}
First we show that $M^+$ contains a strictly convex point  which has positive curvature when $M$ is $\C^2$. Let $\d M^+$ denote the topological boundary of $M^+$ as a subset of $M$, and $\ol{M^+}:=M^+\cup\partial M^+$ denote its closure. $\ol{M^+}$ is compact since $M$ is compact. Let $q$ be a farthest point of $\ol{M^+}$ from $H$. Since $M\cap H$ is compact and $q\not\in H$, there exists a sphere $S$ which contains $M\cap H$ but not $q$. Let $q'$ be a farthest point of $\ol{M^+}$ from the center $o$ of $S$. Then $q'\in M^+$.
 Let $S'$ be the sphere of radius $\|oq'\|$ centered at $o$. Then $M^+$ lies inside $S'$ and intersects it at $q'$.  Hence $q'$ is the desired point.

Now identify $H$ with the hyperplane $x_{n+1}=0$ and suppose after a rescaling that $M^+$ lies in the slab $0 < x_{n+1} < 1$. Consider the projective transformation
\begin{equation}\label{eq:projective}
(x_1, \dots, x_{n},x_{n+1})\overset{P}{\longmapsto} \left(\frac{x_1}{x_{n+1}}, \dots, \frac{x_{n}}{x_{n+1}},\frac{1}{x_{n+1}}\right).
\end{equation}
If $M$ is $\C^2$ with nonnegative curvature, then $P(M^+)$ will be a complete nonnegatively curved hypersurface with a point of positive curvature, since projective transformations preserve sign of curvature. So $P(M^+)$ must be convex by Sacksteder's theorem \cite{sacksteder:convex}, which implies that $M^+$ must have been convex (projective transformations preserve convexity because they preserve line segments). If $M$ is a topological hypersurface which is locally convex, then $P(M^+)$ will be a complete locally convex hypersurface with a strictly convex point. Thus convexity of $P(M^+)$, and subsequently that of $M^+$ follow from the theorem of van Heijenoort \cite{vH:convex}. So we conclude that $M^+$ lies on convex set $K$ with interior points which lie on one side of $H$. Since $M$ is compact, we may assume that $K$ is compact as well.

Since $\ol{M^+}$  is locally embedded and $M^+$ is embedded, it follows that $\ol{M^+}$ is embedded. Since $\ol{M^+}$ is compact, it is  closed in $\d K$. So if $K\cap H$ has no interior points, then $\ol{M^+}=\d K$, which is a contradiction because $\d M\neq\emptyset$ by assumption. So $K\cap H$ must have interior points in $H$. Then the closure of $\d K\cap\inte(H^+)$, which coincides with $\ol{M^+}$, is a convex cap.
\end{proof}

\subsection{Clippings}
If $M$ is locally convex, then through each of its points $p$ there passes a locally supporting hyperplane, i.e., a hyperplane $H$ such that a neighborhood $U$ of $p$ in $M$ lies on one  side of $H$,  where by a \emph{side} we mean  one of the closed half-spaces of $\R^{n+1}$ determined by $H$. We say that $M$ is \emph{one-sided} provided that the side of $H$, say $H^+$, where $U$ lies may be chosen to depend continuously on $p$ (i.e., for every convergent sequence $H_m \to H_\infty$ of supporting hyperplanes of $M$,  we have $H_m^+\to H_\infty^+$). Then $N$ will be called the \emph{inward normal} of $M$, and we say that $M$ is locally convex with respect to $N$. We need to recall the following important fact which is implicit in \cite{alexander&ghomi:chp}:

\begin{lem}[\cite{alexander&ghomi:chp}]\label{lem:project}
Suppose that  $M$ is locally convex, one-sided,  and $\d M$ lies in the interior of a convex body $K$. Then there exists a one-sided locally convex immersed hypersurface $\tilde M$ homeomorphic to $M$ such that $\tilde M$ coincides with $M$ in $K$, while the rest of  $\tilde M$ lies on $\d K$.
\end{lem}
\begin{proof}
For any natural number $k$ there exists a convex polyhedron $P_k$ such that $K\subset P_k$ and the distance between $P_k$ and $K$ is less than $1/k$. For each face of $P_k$, via Lemma \ref{lem:clipping}, clip off the convex caps of $M$ determined by the hyperplane of that face and replace them by flat disks. This yields a sequence of locally convex hypersurfaces $M_k$ which coincide with $M$ in $K$ by \cite[Prop. 4.4] {alexander&ghomi:chp}. The local radii of convexity of $M_k$, as defined in \cite[Sec. 6]{alexander&ghomi:chp}, remain uniformly bounded by \cite[Prop. 6.3 and 6.4] {alexander&ghomi:chp}. Consequently
this sequence converges to the desired surface $\tilde M$ by \cite[Thm. 7.1]{alexander&ghomi:chp}.
\end{proof}

\subsection{Characterizations}
Here are a pair of useful criteria  for checking local convexity, which will be needed below:

\begin{lem}\label{lem:C1tan}
Let $M$ be $\C^1$, $N$ be a continuous normal vector field on $M$, and $T_p M^+$ be the side of $T_p M$ where $N(p)$ points. Suppose that every interior point of $M$ has an open neighborhood $U$ in $M$ which lies in $T_p M^+$. Then the interior of $M$ is locally convex.
\end{lem}
\begin{proof}
Let $B$ be a ball centered at $p$, and $M_p$ be the component of $M$ inside $B$ which contains $p$. Assuming $B$ is sufficiently small, $M_p$ is a disk which meets $\d B$ precisely along its boundary $\d M_p$. By the Jordan-Brouwer theorem, $\d M_p$ separates $\d B$ into a pair of hypersurfaces  $\d B^\pm$ bounded by $\d M_p$. These generate closed embedded hypersurfaces $ M_p\cup\d B^\pm$ of $\R^{n+1}$, which bound compact regions $K^\pm$ respectively. Let $K^+$ be the region into which $N(p)$ points.  Then the interior of $K^+$ forms a connected open set  which is ``weakly supported locally" \cite[Def. 4.8]{valentine:book} at each point of its boundary $\d K^+=M_p\cup\d B^+$. This means that through each point of $\d K^+$ there passes a hyperplanes with respect to which a neighborhood of that point in $K$ lies on one side. Thus, by a theorem of Tietze \cite{tietze1929}, see \cite[Thm. 4.10] {valentine:book}, $K^+$ is convex. So $M_p$ is convex.
\end{proof}

\begin{lem}\label{lem:hessian}
Let $M$ be $\C^2$, and $N$ be a continuous normal vector field on $M$. Suppose that the second fundamental form of $M$ is everywhere positive semidefinite with respect to $N$. Then the interior of $M$ is locally convex.
\end{lem}
\begin{proof}
 Locally $M$ may be represented by graphs of functions over convex sets in the tangent hyperplanes of $M$. These functions will have positive semi-definite Hessians and hence will be convex \cite[Thm. 1.5.13]{schneider2014}.
\end{proof}

\subsection{Regularity}
For the rest of this work, unless stated otherwise, we will assume that $M$ is as in the statement of Theorem \ref{thm:main}.  We need $M$ to be at least $\C^{n+1}$  in order to apply theorems of Sacksteder \cite{sacksteder:convex}, and do Carmo and Warner \cite{docarmo&warner} which analyze the set of flat points of a surface. In particular see  \cite[Lem. 6]{sacksteder:convex}
which requires Sard's theorem \cite[Thm. 3.4.3]{federer:book}, and the subsequent remark \cite[p. 615]{sacksteder:convex}. Otherwise, $\C^2$ regularity would suffice in various lemmas below which do not use these theorems.

\section{Convexity of $\partial M$: Part I}\label{sec:dM}

As we mentioned above, the first step in proving Theorem \ref{thm:main} is to show that every
component $\Gamma$ of $\d M$ is convex in $\S^n$, i.e., it is embedded and bounds a convex set  $X\subset\S^n$. We recall that $X\subset\S^n$ is said to be convex if and only if the cone generated by rays emanating from   $o$ and passing through  points of $X$  forms a convex set in $\R^{n+1}$.
For $n\geq 3$, which we consider first, convexity of $\Gamma$ follows quickly from the free boundary condition, which completely determines the second fundamental form $\II$ of $\d M$ in $M$; specifically, we recall the following observation, which is essentially proved in \cite[Lem. 2]{RV95}. This fact does not depend on the curvature of $M$.

\begin{lem}[\cite{RV95}]\label{lem1}
Let $\nu$ be the outward conormal vector along $\d M$, $p\in \d M$, and $\II$ be the second fundamental form of $\d M$ in $M$ at $p$ with respect to $-\nu$. Then $\nu(p)=\pm p$, and
$$
\II(\cdot,\cdot)=\langle p,\nu(p)\rangle\langle \cdot,\cdot\rangle=\pm \langle \cdot,\cdot\rangle
$$
accordingly, where $\langle \cdot,\cdot\rangle$ is the Euclidean metric. In particular, when $n=2$, the geodesic curvature of $\d M$ in $M$ with respect to $-\nu$ is given by
$$
k(p)=\langle p, \nu(p)\rangle=\pm 1.
$$
\end{lem}

\noindent

Now let $R^{\Gamma}$, $R^{M}$ denote the Riemannian curvature tensors of $\Gamma$ and $M$ respectively, and $\{e_i\}$ be an orthonormal basis for $\Gamma$ at a point $p$. Then by Gauss' equation, and Lemma \ref{lem1},  we may compute that at $p$, for $i\neq j$,
\begin{equation*}
R^{\Gamma}_{ijij}=R^{M}_{ijij}+\II_{ii}\II_{jj}-\II_{ij}\II_{ji}=R^{M}_{ijij}+1-0\geq 1,
\end{equation*}
where subscripts indicate the coefficients of  these tensors with respect to $\{e_i\}$. So the sectional curvatures of $\Gamma$ are bounded below by $1$. Thus, by the theorem of do~Carmo and Warner \cite{docarmo&warner}, $\Gamma$ is convex in $\S^n$ when $n\geq 3$.

It remains then to consider the case where $n=2$, which will be significantly more involved, because a locally convex closed curve in $\S^2$ need not be globally convex, or even embedded. The arguments below will depend on whether $M$ lies outside or inside $\S^2$ near $\Gamma$, and will be presented in the next two sections respectively.  Note that $\nu(p)=p$ whenever $M$ meets $\S^n$ from the inside, and $\nu(p)=-p$ whenever $M$ meets $\S^n$ from the outside.

\section{Convexity of $\d M$: Part II} \label{sec:dMII}
Throughout this section we will assume that $n=2$, and $M$ lies outside $\S^2$ near a component $\Gamma$ of $\d M$.
In order to establish the convexity of $\Gamma$ in this case, we will have to show that  $M$ is locally convex. To start, let $U$ be a tubular neighborhood of $\Gamma$ in $M$. Assuming $U$ is small, $U\setminus\Gamma$ will lie outside of $\S^2$. Let $M_\Gamma$ be the closure of the component of $M$ outside of $\S^2$ which contains $U\setminus\Gamma$. We claim that $M_\Gamma$ is locally convex and one-sided, as defined in Section \ref{sec:prelim}. To this end first we show:

\begin{lem}\label{lem:inteM_Gamma}
There exists a unique continuous unit normal vector field $N$ on $M_\Gamma$ with respect to which the  interior of $M_\Gamma$ is locally convex.
\end{lem}
\begin{proof}
By Lemma \ref{lem:clipping}, each point $p\in\inte(M_\Gamma)$ lies in the interior of a convex cap $C_p$. Let $N(p)$ be the unit normal vector of $M$ at $p$ which points to the side of $T_p M$ where $C_p$ lies. To see that $N(p)$ is well defined, i.e., it does not depend on the choice of a cap, let $C_p'$ be another convex cap which contains $p$ in its interior. Suppose, towards a contradiction, that $C_p$ and $C_p'$ lie on opposite sides to $T_p M$. Then $C_p\cap C_p'$ must lie in $T_p M$. In particular neither cap can lie completely inside the other, for else it would have to be flat, which is not possible. It follows then that $\d C_p\cap \d C_p'$ must contain at least a pair of points. Let $L$ be the line passing through these points. Furthermore,  let $H$, $H'$ be the planes on which $\d C_p$, $\d C'_p$ lie respectively. Then $H$, $H'$ both must contain $L$. On the other hand, $H$, $H'$ cannot coincide with $T_p M$. Thus $H\cap T_p M=L=H'\cap T_p M$. Consequently, $\d C_p\cap \d C'_p$ forms a line segment in $L$. This again would imply that one cap lies inside the other, which is impossible as we pointed out earlier.  So $N(p)$ is well-defined, as claimed. Next note that $N$ is continuous on $\inte(M_\Gamma)$, because it is continuous on each $C_p$.  Finally, we may extend $N$ continuously to the boundary of $M_\Gamma$, since as we mentioned in Section \ref{subsec:terminology}, we may assume that $M$ is orientable. More explicitly, there exists a continuous unit normal vector field $\nu$ on $M_\Gamma$. After replacing $\nu$ with $-\nu$, we may assume that $N=\nu$ on $\inte(M_\Gamma)$, since $\inte(M_\Gamma)$ is connected. Then we set $N=\nu$ on $\d M_\Gamma$ which completes the proof.
\end{proof}

Next we consider the local convexity of $M_\Gamma$ along $\Gamma$. To this end we need to study the behavior of $\Gamma$ in $\S^2$. For the rest of this section, unless indicated otherwise, $N$ will be the vector field given by the last lemma.

\begin{lem}\label{lem:localconvexgamma}
The geodesic curvature of $\Gamma$ in $\S^2$ is  nonnegative with respect to $N$.
\end{lem}
\begin{proof}
For every $p\in \Gamma$, $N(p)$ is normal to $\Gamma$ and tangent to $\S^2$. Thus it follows that the geodesic curvature of $\Gamma$ at $p$ is given by $k(p)=\II_p(T,T)$, where $\II_p$ is the second fundamental form of $M$ at $p$ with respect to $N$, and $T$ is a unit tangent vector of $\Gamma$ at $p$.  Take a sequence of points $p_i$ in $U\setminus \Gamma$ converging to $p$, and let $T_i\in T_{p_i}M$ be a sequence of unit tangent vectors converging to $T$. Then $\II_{p_i}(T_i, T_i)$ converges to $\II_p(T,T)$, since $M$ is $\C^2$. By Lemma \ref{lem:inteM_Gamma}, $\II_{p_i}(T_i,T_i)$ is nonnegative. So $\II_p(T,T)$ is nonnegative, which yields that $k\geq 0$ on $\Gamma$ as desired.
\end{proof}

To establish the convexity of $\Gamma$ it only remains  to check that it is simple. Indeed any simple spherical curve whose geodesic curvature is nonnegative with respect to a continuous normal vector field must be convex \cite[Lem. 2.2]{ghomi:verticesC}.

In the next lemma we need to apply the Gauss-Bonnet theorem to a nonsmooth surface. For this purpose we choose the theorem in  the book of Alexandrov and Zalgaller \cite[p. 192]{alexandrov-zalgaller} which mirrors the traditional version of the Gauss-Bonnet theorem, and applies to Alexandrov surfaces, i.e., $2$-dimensional manifolds with a metric whose curvature is bounded in the sense of Alexandrov. With the induced metric, all $\C^2$ surfaces immersed in $\R^3$ are examples of these objects, as are all locally convex surfaces, whose curvature in the sense of Alexandrov is nonnegative.

\begin{lem}\label{lem:disk}
 $M$ is homeomorphic to a disk.
 \end{lem}
 \begin{proof}
 Let $\Gamma_1$, $\Gamma_2, \dots$ denote those components of $\d M$ near which $M$ lies outside $\S^2$. For each $i$ let  $U_i$ be a tubular neighborhood of $\Gamma_i$ in $M$. Let $S$ be the sphere of radius $1+\epsilon$ centered at $o$, and for each $i$ set $\Gamma_i':=U_i\cap S$. Choosing $\epsilon$ sufficiently small, we may suppose that $S$ meets every $U_i$ transversally so that $\Gamma_i'$ is a smooth closed curve. For all $i$, let $A_i$ be the annular region bounded by $\Gamma_i$ and $\Gamma_i'$.

By Lemma \ref{lem:clipping}, the parts of $M$ which lie outside $\S^2$ are locally convex. Thus we may project these parts  into $S$ via Lemma \ref{lem:project}. More specifically, by perturbing a sphere of radius $1+\epsilon/2$ between $\S^2$ and $S$, we obtain a closed surface $C$ which meets $M$ transversely, by the transversality theorem \cite{hirsch:book}. Then portions $M_C$ of $M$ which lie outside $C$ are manifolds whose boundaries lie strictly inside $S$. So we may apply Lemma \ref{lem:project} to $M_C$ with respect to  the convex body $K$ bounded by $S$. This results in an immersed surface $\tilde M$ homeomorphic to $M$ which coincides with $M$ inside $S$.

Let $\tilde M'$ be the closure of the surface obtained from $\tilde M$ be cutting off the annular regions $A_i$. Then $\tilde M'$ is homeomorphic to $M$, and  lies in $S$ near each $\Gamma_i'$. Let $N'$ be the inward conormal vector of $\d \tilde M'$ in $\tilde M'$ along  $\Gamma_i'$. Note that  for each $p\in\Gamma_i'$, $N'(p)$ points to the side of $T_p M$, say $T_p M^+$, where the inward normal $N(p)$ of $M$ points, because by Lemma \ref{lem:project} $\tilde M$ is one sided and coincides with $M$ inside $S$. Indeed $N$ is  the inward normal of $\tilde M$ on $A_i$, and so there exists an open neighborhood $U$ of $p$ in $\tilde M$ which lies in $T_p M^+$. In particular $U\cap \tilde M'$, which is an open neighborhood of $p$ in $\tilde M'$, lies in $T_p M^+$.  This shows that as $\epsilon\to 0$,   $N'$ converges to $N$. Thus the integral of the geodesic curvature of $\Gamma_i'$ in $\tilde M'$ with respect to $N'$ converges to the integral of the geodesic curvature of $\Gamma_i$ in $\S^2$ with respect to $N$, which is nonnegative by Lemma \ref{lem:localconvexgamma}.

Furthermore, along those boundary components of $\tilde M'$ where $\tilde M'$ meets $\S^2$ from the inside, the geodesic curvature is  positive by Lemma \ref{lem1}. Thus, for $\epsilon$ sufficiently small, the integral of geodesic curvature of $\d \tilde M'$ in $\tilde M'$ will be positive or else arbitrarily close to zero. In addition note that every point of $\tilde M'\subset\tilde M$ is either $\C^2$ and nonnegatively curved or else is locally convex, in which case its curvature is still nonnegative everywhere in the sense of Alexandrov \cite{alexandrov-zalgaller}. In addition, since parts of $\tilde M'$ coincide with $S$, its total curvature is positive and remains bigger than some positive constant as  $\epsilon\to 0$.  Thus, by the Gauss-Bonnet theorem for Alexandrov surfaces \cite[p. 192]{alexandrov-zalgaller}, $\tilde M'$ is a disk. So $M$ is a disk.
\end{proof}

Note that the last lemma implies in particular that $\d M$ is connected and so $\Gamma=\d M$. For the rest of this section we will use $\d M$ and $\Gamma$ interchangeably. We now can show:

\begin{lem}\label{lem:localconvgamma}
$M$ is locally convex along $\d M$ with respect to $N$.
\end{lem}
\begin{proof}[Proof]
By Lemma \ref{lem:disk}, $\d M=\Gamma$. So, by Lemma \ref{lem:inteM_Gamma}, there exists an open neighborhood $U$ of
$\d M$ in $M$  such that $U\setminus\d M$ is locally convex with respect to $N$. We claim that  for every $p\in \d M$, $M$ lies locally on the side of $T_p M$, say $(T_p M)^+$, where $N(p)$ points.  This would complete the proof as follows. Extend $U$ to a larger surface $\tilde U$, by attaching to each point $p$ of $\d M$ a portion of the segment $op$, say of length $1/2$. Then $\tilde U$ satisfies the hypothesis of Lemma \ref{lem:C1tan}, assuming that the claim holds. Note that $\tilde U$ is $\C^1$ because it has  flat tangent cones at each point which vary continuously, see \cite[Lem. 3.1]{ghomi&howard:tancones}. Thus $\tilde U$ is locally convex, as desired. It remains then to establish the claim.

By Lemma \ref{lem:localconvexgamma},  there exists a simple segment of $\d M$, say $\Gamma_0$, which contains $p$ in its interior and lies in $(T_p M)^+$. Let $C$ be the  surface generated by rays which originate from $o$ and pass through all points of $\Gamma_0$. Then $C$ lies in $(T_p M)^+$. So to establish the claim it suffices to show that, near $p$, $M$ lies on the side of $C$
\begin{figure}[h]
\centering
\begin{overpic}[height=1.7in]{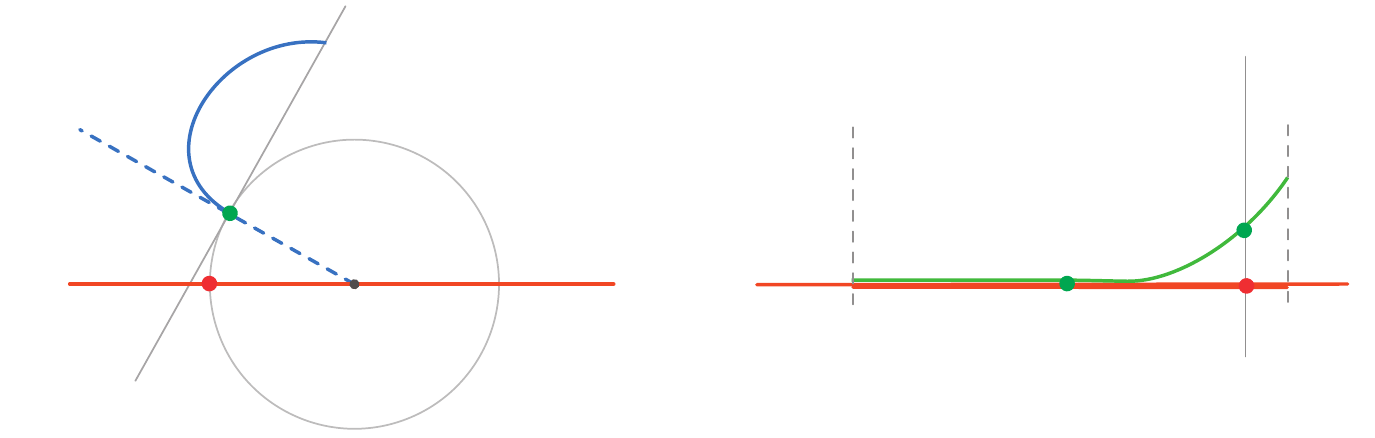}
\put(25, 9.5){\Small $o$}
\put(16,9.5){\Small $q$}
\put(16.5,14.5){\Small $q'$}
\put(77,9){\Small $p$}
\put(91,9.5){\Small $q$}
\put(91,14){\Small $q'$}
\put(1, 9){\Small $T_p M\cap H_q$}
\put(53, 8.75){\Small $G$}
\put(86,4){\Small $H_q\cap\S^2$}
\put(85, 14){\Small $\Gamma_0$}
\put(65, 8.5){\Small $G_0$}
\put(9,28.5){\Small $X_q\cap H_q$}
\put(2,23){\Small $C\cap H_q$}
\put(3,2){\Small $T_{q'}\S^2\cap H_q$}
\end{overpic}
\caption{}\label{fig:circle}
\end{figure}
where $N$ points.

Note that, by the free boundary condition, $T_p M$ passes through $o$ and so $G:= T_p M\cap\S^2$ is a great circle. We may assume that $\Gamma_0$ is a graph over a segment $G_0$ of $G$, i.e., every great circle orthogonal to $G_0$ intersects $\Gamma_0$  at most once and every point of $\Gamma_0$ lies on such a circle; see the right diagram in Figure \ref{fig:circle}.
Let $q\in G_0$, and $q'$ be the corresponding point in $\Gamma_0$, i.e,  the intersection with $\Gamma_0$ of the great circle which  is orthogonal to $G_0$ at $q$.

Let $(T_{q'}\S^2)^+$ be the side of $T_{q'}\S^2$ which does not contain $\S^2$. By Lemma \ref{lem:clipping}, $q'$ lies on the boundary of a convex cap $X_q$ in $(T_{q'}\S^2)^+$, with boundary on $T_{q'}\S^2$. Note that the ray $oq'$ of  $C$ does not intersect the interior of the convex hull $K_q$ of $X_q$, because $X_q$ is tangent to $T_{q'} M$, and hence lies on one side of $T_{q'} M$, while $oq'$ lies in $T_{q'} M$.
Now let $H_q$ be the plane which passes through $o$, $q$, and $q'$ and is orthogonal to $T_pM$. Since $H_q$ is transverse to $\Gamma_0$ and $\d X_q$ is tangent to $\Gamma_0$ at $q'$, we may suppose that $H_q$ intersects $X_q$ transversally.
So $K_q$ has interior points in $H_q$ and thus $H_q\cap K_q$ is a convex body in $H_q$. Consequently $X_q\cap H_q$ is a convex cap which lies on the boundary of $K_q\cap H_q$. In particular, since $oq'$ does not intersect the interior of $K_q$,  $X_q\cap H_q$ lies on one side of $oq'$ in $H_q$.

Note that the curves $X_q\cap H_q$ fibrate an open neighborhood $U$ of $p$ in $M$. Indeed we may take $U$ to be the union of the interior of $\Gamma_0$ with the interior of all caps $X_q$ for $q$ in the interior of $G_0$. Thus, since $X_q\cap H_q$ lies on one side of $oq'=C\cap H_q$, it follows that $U$ lies on one side of $C$, as desired. Finally, we check that this is the side of $C$ where $N$ points. To see this recall that, by Lemma \ref{lem:inteM_Gamma}, $N$ is the inward normal in the interior of $M$ near $\d M$. In particular, $N$ is the inward normal on the interior of each cap $X_q$. By continuity it follows that $N$ is the inward normal on all of $X_q$. Thus $X_q$ lies in the side of $T_{q'} M$ where $N(q')$ points, which yields that $U$ lies on the side of $C$ where $N$ points, and completes the proof.
\end{proof}

The principal step \cite[Thm. 1]{sacksteder:convex} in the proof of Sacksteder's theorem  is that each component  of the set of flat points of a complete nonnegatively curved hypersurface is a convex planar set; see also \cite[Lem. 3.1]{alexander&ghomi:chp} and \cite[p. 460]{greene&wu:rigidityII}. Here we need an analogue of this fact for surfaces with boundary, which constitutes the key observation in this section:

\begin{lem}\label{lem:localconvgammaII}
Let $p\in\d M$, and $X$ be the component of $T_p M\cap M$ which contains $p$. Then $X$ is fibrated by line segments which meet $\partial M$ orthogonally. In particular $M\setminus X$ is connected.
\end{lem}
\begin{proof}
We will use the same setting and notation as in the proof of Lemma \ref{lem:localconvgamma}, and refer the reader to Figure \ref{fig:circleII}, which adds new details to Figure \ref{fig:circle}.  In particular, an important tool will be the fibration $X_q\cap H_q$ of the neighborhood of $\Gamma_0$ in $M$.

\begin{figure}[h]
\centering
\begin{overpic}[height=1.7in]{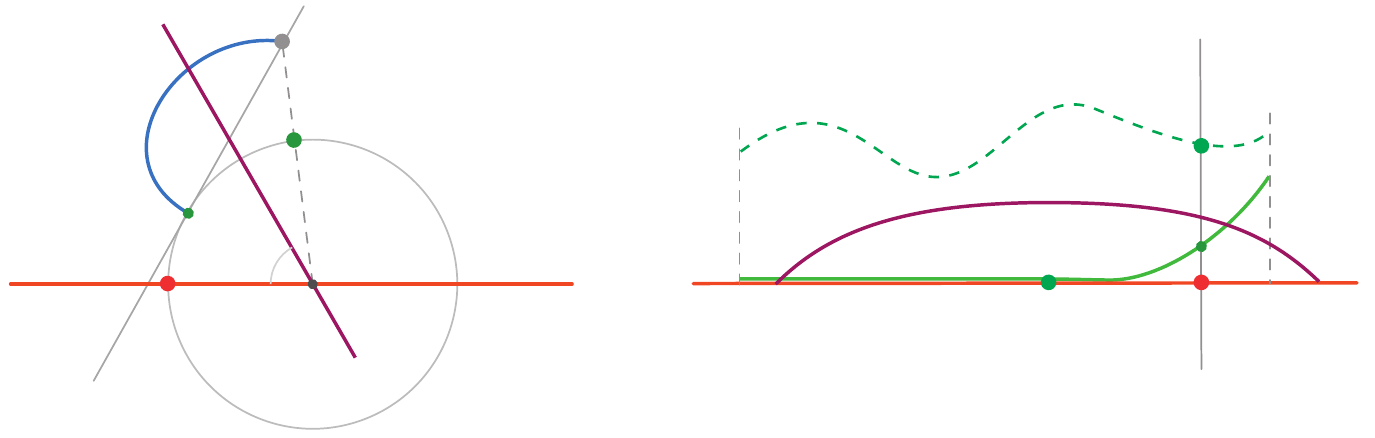}
%\put(21.5, 9.5){\Small $o$}
\put(18.75, 12.5){\Small $\theta$}
\put(12,9){\Small $q$}
\put(13.5,14.5){\Small $q'$}
\put(21.5,28){\Small $q''$}
\put(22.5,19.5){\Small $\ol q''$}
%\put(15,26){\Small $q'''$}
\put(75.75,9.25){\Small $p$}
\put(88,9.5){\Small $q$}
\put(88,13){\Small $q'$}
\put(88,22){\Small $\ol q''$}
\put(-2, 9){\Small $T_p M\cap H_q$}
\put(22.5, 4){\Small $H\cap H_q$}
\put(49, 8.75){\Small $G$}
\put(55, 15){\Small $H\cap\S^2$}
\put(83,3){\Small $H_q\cap\S^2$}
\put(83, 13.5){\Small $\Gamma_0$}
%\put(65, 8.5){\Small $G_0$}
%\put(2,24){\Small $X_q\cap H_q$}
%\put(1,2){\Small $T_{q'}\S^2\cap H_q$}
\end{overpic}
\caption{}\label{fig:circleII}
\end{figure}

Let $H$ be a plane different from $T_p M$ which passes through $o$, and such that the line $H\cap T_p M$ is orthogonal to $op$. Let $H^+$ be the side of $H$ where $p$ lies, $H^-$ be the opposite side, and $\theta$ be the angle of the wedge $H^+\cap (T_p M)^+$. By Lemma \ref{lem:localconvexgamma} we can make sure $\Gamma_0$ is long enough so that each end point of $\Gamma_0$ lies either in the interior of $(T_p M)^+$ or in the interior of $H^-$. Then, choosing $\theta$  sufficiently small, we may assume that both end points of $\Gamma_0$ lie  in the interior of $H^-$.  Let $M^+$  be the closure of the component of $M$ which lies in the interior of $H^+$ and contains $p$.
We claim that if
 $\theta$ is sufficiently small, then $M^+$ is locally convex.

To establish the claim first note that $M^+$ is locally convex along $\d M$ by Lemma \ref{lem:localconvgamma}. Thus it suffices to check the points of $M^+$ which lie in the interior of $M$. To this end let $q''$ be the end point of $X_q\cap H_q$, other than $q'$, and $\ol q''$ be the projection of $q''$ into $\S^2$. Furthermore let $d$ be the smallest geodesic distance between $\ol q''$  and $q$ in $\S^2$ for all $q\in G_0$, $d:=\inf_{q\in G_0}\dist_{\S^2}(q, \bar{q}'')$. Note that $d>0$, because the planes $T_{q'}\S^2$ which determine the convex cap $X_q$, are transversal to $M$ along $\d X_q$. So $\d X_q$ depends continuously on $q$, as $T_{q'}\S^2$ depends continuously on $q$. Furthermore, $H_q$ depends continuously on $q$ as well. Hence $q''$ depends continuously on $q$, since it is one of the two points where $H_q$ and $\d X_q$ meet. So $\dist_{\S^2}(q, \bar{q}'')$ depends continuously on $q$.
Finally, note that $q''\neq q'$ which yields that $\ol q''\neq q'$. Consequently $\ol q''\neq q$, since $q'$ lies in  the geodesic segment $q\ol q''$. So $\dist_{\S^2}(q, \bar{q}'')>0$, which yields that $d>0$ due to compactness of $G_0$. Now setting $\theta<d$ yields the desired angle, for then $M^+\cap\inte(M)$ is covered by the interior of convex caps.

Having established the local convexity of $M^+$, we now let $\ol {M^+}$ be the extension of $M^+$ which is obtained by connecting points of $M^+\cap \d M$ to $o$. Then $\ol {M^+}$ is a locally convex surface whose boundary lies in $H$, and therefore is a convex cap by Lemma \ref{lem:clipping}.
Let  $Y_p$ be the component of $X$ containing $p$ which lies in $H^+$, and $\ol Y_p$ be the extension of $Y_p$ obtained by connecting all points of $Y_p\cap\d M$ to $o$; see Figure \ref{fig:pie}. Then $\ol Y_p=\ol{M^+}\cap T_p M$. Thus $\ol Y_p$ is a convex set. In particular, for any point $x\in  Y_p$, the segment $ox$ is contained in $\ol Y_p$. Let $x'$ be the intersection of $ox$ with $\partial M$, and extend $ox$ until it intersects the boundary of $Y_p$ at another point, say $x''$. Then the segment $x'x''$ lies in $Y_p$, and thus we obtain a fibration of $Y_p$ by line segments orthogonal to $\partial M$.

\begin{figure}[h]
\centering
\begin{overpic}[height=1.2in]{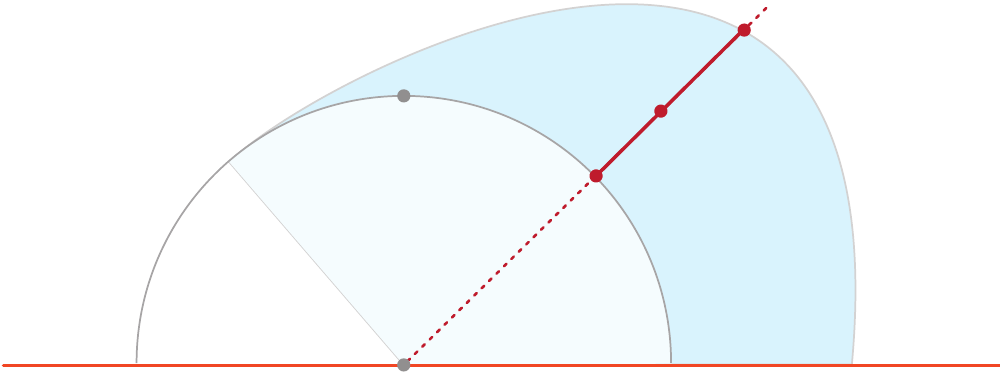}
\put(39.5, -2){\Small $o$}
\put(36, 11){\Small $\ol Y_p$}
\put(75, 19){\Small $Y_p$}
\put(59.5, 16.75){\Small $x'$}
\put(67,24){\Small $x$}
\put(76,33){\Small $x''$}
\put(40,25){\Small $p$}
\put(-17,-2){\Small $T_p M\cap H$}
\put(-4,18){\Small $T_p M\cap H^+$}
\end{overpic}
\caption{}\label{fig:pie}
\end{figure}

Note that the above construction may be carried out for any point $r$ of $X\cap\d M$ to yield a fibrated set $Y_r\subset X$ for each $r$. Thus, to complete the proof, it suffices to show that the sets $Y_r$ cover $X$. To this end we need to check that $Y:=\cup_{r\in X\cap\d M} Y_r$ is open and closed in $X$. To establish the closedness, let $x_i\in Y$ be a sequence of points converging to a point $x$ of $X$. Then $x_i'$ converge to $x'$. By assumption, the segments $x_i x_i'\subset Y_{x_i'}\subset X$. Thus, as $X$ is closed, $xx'$ lies in $X$. Consequently $x\in xx'\subset Y_{x'}\subset Y$. So $Y$ is indeed closed in $X$. It remains to check then that $Y$ is open in $X$. To see this let $x\in Y$. Then $x\in Y_{x'}$. Let $V$ be an open neighborhood of $x$ in $X$. We may assume that $V$ is connected and is so small as to be contained in the half-plane in $T_p M$ determined by the line orthogonal to $ox$, which passes through $o$. Recall that $Y_{x'}$ is by definition the connected component of $X$, containing $x'$, which lies in that half-plane. It follows then that $V\subset Y_{x'}\subset Y$, since $V\cup Y_{x'}$ is connected and lies in the half-plane. So $Y$ is open in $X$, and we are done.
\end{proof}

Now we are ready to prove the main result of this section:

\begin{lem}\label{lem:localconvex}
$M$ is locally convex.
\end{lem}
\begin{proof}
By Lemma \ref{lem:localconvgamma}, $M$ is locally convex along $\d M$. So it remains  to check that the interior of $M$ is locally convex as well. To this end,  by Lemma \ref{lem:hessian}, it suffices to show that the second fundamental form of $M$ is positive semidefinite with respect to a continuous normal vector field. We claim that the desired vector field is given by $N$ once it is extended to all of $M$. By Lemma \ref{lem:localconvgamma}, the second fundamental form of $M$ with respect to $N$ will then be positive semidefinite in a connected neighborhood $U$ of $\d M$. Following Sacksteder \cite{sacksteder:convex}, we let $M_0$ be the set of flat points of $M$, and $M_1:=M\setminus M_0$. Note that each component of $M_1$ admits a unique choice of inward normal. Hence it suffices to show that each component of $M_1$ intersects $U$, or that $U\cup M_1$ is connected. Equivalently we need to show that no component of $M_0$ separates a component of $M_1$ from $U$. By \cite[Thm. 1]{sacksteder:convex}, if a component $X$ of $M_0$ lies in the interior of $M$ , then it is a convex planar set. So $M\setminus X$ is connected. On the other hand, if $X$ intersects $\d M$,  then again $M\setminus X$ is connected by Lemma \ref{lem:localconvgammaII}. Thus no component of $M_0$ separates $M$, which completes the proof.
\end{proof}

To establish the embeddedness, or simplicity, of $\Gamma$ we need only one more observation concerning general properties of spherical curves. The following lemma applies to all curves in $\S^2$ whose geodesic curvature is nonnegative with respect to a continuous normal vector field.

\begin{lem}\label{lem:hemisphere}
If $\Gamma$ is not simple, then either it  traces a great circle multiple times, or else it contains a subloop which lies in an open hemisphere.
\end{lem}
\begin{proof}
If the curvature of $\Gamma$ is identically zero, then it traces a great circle and there is nothing to prove. Suppose then that $\Gamma$ has a point $p$ of nonzero curvature. Let $C$ be the great circle passing through $p$ and tangent to $\Gamma$ at $p$. Then  a neighborhood of $p$ in $\Gamma$ lies inside $C$, i.e., in the hemisphere $H$ bounded by $C$ where $N(p)$ points, and intersects $C$ only at $p$. If $\Gamma$ intersects $C$ at no point other than $p$,  we may  slightly shift $C$ to make it disjoint from $\Gamma$. Then $\Gamma$ will lie in an open hemisphere and we are done. So we may assume that $\Gamma$ intersects $C$ at some point
\begin{figure}[h]
\centering
\begin{overpic}[height=1.25in]{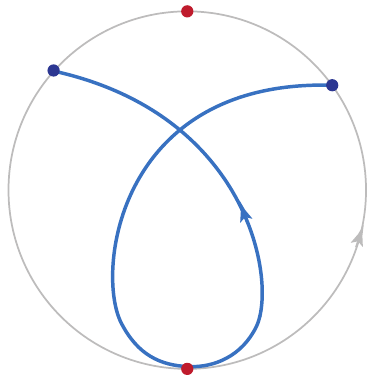}
\put(47,-6){\Small $p$}
\put(48,102){\Small $p'$}
\put(6,81){\Small $q$}
\put(90,81){\Small $r$}
\put(90,15){\Small $C$}
\end{overpic}
\caption{}\label{fig:loop}
\end{figure}
other than $p$. Orient $\Gamma$ and $C$ so that their orientations coincide at $p$, see Figure \ref{fig:loop}. Let $q$ be the first point in $\Gamma$ after $p$ where $\Gamma$ intersects $C$. We may assume that $p\neq q$, and the interior of $pq$ is simple for otherwise we are done (in the first case we obtain a loop intersecting $C$ at only one point, and in the second case we obtain a loop contained entirely in the interior of $H$). Then, by \cite[Lem. 3.1]{ghomi:verticesC}, $q$ must lie in the interior of the (oriented) segment $p'p$ of $C$, where $p':=-p$. Similarly let $r$ be the first point, as we traverse $\Gamma$ from $p$ against its orientation, that lies on $C$. Again we may assume that $rp$ is simple and $r\neq p$. Then   \cite[Lem. 3.1]{ghomi:verticesC} implies that $r$ must lie in the interior of the segment $pp'$. Hence the segments $pq$ and $rp$ must intersect in the interior of $H$. This yields a loop which lies inside $C$ and intersects $C$ only at $p$. So it must lie in an open hemisphere.
\end{proof}

Now we can show that $\Gamma$ is simple. First  connect all points of $\Gamma$ to $o$ by straight line segments. This extends $M$ to a closed surface $\ol M$. It follows from Lemma \ref{lem:localconvex} that $\ol M\setminus\{o\}$ is locally convex.
Suppose that $\Gamma$ is not simple. Then by  Lemma \ref{lem:hemisphere} either (i) $\Gamma$ is a multiple covering of a great circle or (ii) $\Gamma$ has a subloop which lies in an open hemisphere.

If $\Gamma$ multiply covers a great circle, then as in the proof of Lemma \ref{lem:disk}, we may apply Lemma \ref{lem:project} to $\ol M\setminus\{o\}$, and with respect to $K=B^{3}$, to obtain a compact locally convex surface $\tilde M$ homeomorphic to $M$, which is bounded by $\Gamma$ and coincides with $\S^2$ in a neighborhood of $\Gamma$. Then the total geodesic curvature of $\d \tilde M$ is zero, and so Gauss-Bonnet theorem for Alexandrov surfaces implies that the total curvature of $\tilde M$ must be $2\pi$. On the other hand, the Gauss map of $\tilde M$ sends $\d \tilde M$ to a multiple covering of a great circle in $\S^2$, which implies that the total curvature should be bigger than $2\pi$ and we obtain the desired contradiction.

So we may assume that $\Gamma$ has a subloop, say $\Gamma_0$, which lies in an open hemisphere.
 Then there exists a plane $H$ which separates $\Gamma_0$ from $o$. Let $X$ be the component of $\ol M$ which contains $\Gamma_0$ and lies on the side of $H$ not containing $o$. Then $X$ is convex and is therefore embedded by Lemma \ref{lem:clipping}. This is a contradiction because $X$ contains a double point by construction. Hence $\Gamma$ is indeed simple. As we had mentioned earlier, this together with Lemma \ref{lem:localconvexgamma} completes the proof of the convexity of $\Gamma$, due to the characterization for convex spherical curves in  \cite[Lem. 2.2]{ghomi:verticesC}.

\begin{note}
Trying to establish the convexity of $\d M$ in this section, we had to prove that $M$ is locally convex. These two facts now yield the convexity of $M$. Indeed connecting points of $\d M$ to $o$ by line segments yields a closed locally convex surface $\ol M$. By van Heijenoort's theorem, $\ol M$ is convex. Thus $M$ is convex. Further recall that $M$ is a disk by Lemma \ref{lem:disk}. So we have proved Theorem \ref{thm:main} in the case where $n=2$ and $M$ lies outside $\S^2$ near one of its boundary components.
\end{note}

\section{Convexity of $\d M$: Part III}\label{sec:dMIII}
To complete the proof of the convexity of the components of $\d M$ it remains to consider the case where $n=2$ and
$M$ lies inside $\S^2$ near a component $\Gamma$ of $\d M$, which we assume is the case throughout this section.
If $M$ meets any one of its boundary components from outside $\S^2$, then by Lemma \ref{lem:disk} that is the only boundary component it has and convexity of $\d M$ follows from the last section. So we may further assume that $M$ meets $\S^2$ from the inside along all of its boundary components. Recall that, by Lemma \ref{lem1}, the geodesic curvature $k\equiv 1$ on $\partial M$. Thus,
by the Gauss-Bonnet Theorem, we have
$$
2\pi \chi(M)=\int_{\d M}k +\int_{M} K=\Length[\d M] +\int_{M} K\geq \Length[\d M]>0.
$$
So the Euler characteristic $\chi(M)>0$, which means $M$ is a topological disk. In particular $\d M$ is connected, and so $\Gamma$ is the only boundary component of $M$. Thus
\begin{equation}\label{eq:Gauss-Bonnet}
\Length[\Gamma] +\int_{M} K=2\pi.
\end{equation}
The rest of the argument will be divided into two parts: (i) $K\not\equiv 0$, and (ii) $K\equiv 0$:

\subsection{($K\not\equiv 0$)} If $K$ is not identically zero, then $\Length[\Gamma] <2\pi$ by \eqref{eq:Gauss-Bonnet}. Furthermore, by Crofton's formula,
\begin{equation*}
\Length[\Gamma]=\frac{1}{4}\int_{p\in\S^2}\#(p^\perp\cap \Gamma)d\sigma,
\end{equation*}
where $p^\perp$ denotes the oriented great circle centered at $p$.
It follows that $\Gamma$ misses some great circle in $\S^2$, and therefore lies in an open hemisphere. In particular $\Gamma$ has a well-defined convex hull in $\S^2$ (given by the intersection of all closed hemispheres which contain $\Gamma$). Let $\Gamma'$ be the boundary of that convex hull, and let $M'$ be the convex  surface obtained by connecting $o$ to points of $\Gamma'$ with straight line segments. Since $M'$ is orthogonal to $\S^n$ along $\Gamma'$, again Lemma \ref{lem1} yields that the geodesic curvature $k'$ of $\Gamma'$ in $M'$ is identically one. Thus $\int_{\Gamma'}k'=\Length[\Gamma']$.
Then by the Gauss-Bonnet theorem for Alexandrov surfaces \cite[p. 192]{alexandrov-zalgaller} (recall the discussion prior to Lemma \ref{lem:disk}),
 \begin{equation}\label{eq:olGamma}
\Length[\Gamma'] +\int_{M'} K'=2\pi.
\end{equation}
But $\Length[\Gamma']\leq \Length [\Gamma]$. Further we claim that $\int_{M'}K'\leq \int_{M} K$. Then comparing \eqref{eq:Gauss-Bonnet} and \eqref{eq:olGamma} would yield that $\Length[\Gamma']= \Length [\Gamma]$, which may happen only if $\Gamma'=\Gamma$. Hence $\Gamma$ will be convex as desired.
So it remains only to check that $\int_{M'}K'\leq \int_{M} K$. To  this end first we show that

\begin{lem}\label{lem:mm'}
Every support plane of $M'$ passing through $o$ is parallel to a tangent plane of $M$.
\end{lem}
\begin{proof}
Let $u$ be the outward to  a support plane of $M'$ at $o$, i.e., $\langle u,p\rangle\leq 0$ for all $p\in M'$.
Let $H_t$ be the plane orthogonal to $u$  which passes through the point $tu$, and $t_0$ be the infimum of $t\in\R$ such that $H_t\cap M=\emptyset$. Then $H_0:=H_{t_0}$ intersects $M$ at a point $p$, while $M$ lies in the side of $H_0$, say ${H_0}^+$, which is opposite to where $u$ points. If $p$ is in the interior of $M$, then $H_0=T_p M$ and we are done. Suppose then that $p\in\Gamma$. So $t_0\leq 0$. Then, since $\Gamma\subset {H_0}^+$, we have
$$
\langle p, u\rangle\leq 0.
$$
On the other hand, if $\nu$ denotes the inward conormal of $M$ along $\d M$, then we also have $\langle \nu(p),u\rangle=\frac{d}{d\tau}\big|_{\tau=0+}\langle p+\tau\nu(p),u\rangle \leq 0$, since $M\subset {H_0}^+$. But $\nu(p)=-p$ since by assumption $M$ lies inside $\S^2$ near $\Gamma$. Thus
$$
\langle p,u\rangle=\langle-\nu(p),u\rangle\geq 0.
$$
So we conclude that $\langle p, u\rangle=0$, which means that $H_0$ is orthogonal  to $\S^2$. Since $H_0$ is tangent to $\Gamma$ at $p$, it follows that $H_0=T_p M$ as desired.
\end{proof}

Let $N$ be a normal vector field for $M$, and $N'$ be the outward unit normal map of $M'$ ($N'$ is multivalued at $o$). Since $M'$ is convex,
$$
\int_{M'} K'=\sigma (N'(M')),
$$
where $\sigma$ is the area measure in $\S^2$. Further, since $\Gamma'$ lies in an open hemisphere, $N'(M')$ lies in an open hemisphere as well; because every vector in $N'(M')$ is the outward normal to a plane which passes through $o$ and supports $M'$. Thus, if $\pi\colon\S^2\to\RP^2$ is the standard projection,
$$
\sigma (N'(M'))=\ol\sigma(\pi(N'(M'))),
$$
where $\ol\sigma$ denotes the area measure in $\RP^2$.
By Lemma \ref{lem:mm'}, for every $u'\in N'(M')$, there exists $u\in N(M)$ such that $u'=\pm u$, or $\pi(u)=\pi(u')$. So $\pi(N'(M'))$ is covered by $\pi(N(M))$. In particular
$$
\ol\sigma(\pi(N'(M')))\leq\ol\sigma(\pi(N(M))).
$$
Finally note that, for any set $X\subset\S^2$, $\ol\sigma(\pi(X))\leq \sigma(X)$. Thus
$$
\ol\sigma(\pi(N(M))) \leq \sigma(N(M))\leq \int_{M} \det(dN)=\int_M K.
$$
The last four displayed expressions yield that $\int_{M'} K'\leq \int_M K$ as desired. So we conclude that $\Gamma$ is convex when $K\not\equiv 0$.

\subsection{($K\equiv 0$)} If $K$ vanishes identically, then by \eqref{eq:Gauss-Bonnet},
$$
\Length[\d M] =2\pi.
$$
Next we need the following lemma concerning the structure of an immersed disk with zero Gauss curvature in $\R^3$, which is due to Hartman and Nirenberg \cite[Thm. A]{hartman&nirenberg}; see also Massey \cite{Mas62} or do Carmo \cite[Sec. 5.8]{docarmo:csbook}.

\begin{lem}[\cite{hartman&nirenberg}]\label{lem3}
Let $D$ be a $\C^2$ disk of zero Gauss curvature immersed in $\R^3$. Then every point of $D$ either lies on a line segment in $D$ with end points on the boundary of $D$, or lies on a planar domain in $D$ whose boundary consists of line segments with end points on $\d D$ or arcs of $\d D$.
\end{lem}

By Lemma \ref{lem3} and the free boundary condition, all tangent planes of $M$ must go through $o$. It follows that $M$ itself must go through $o$, otherwise the tangent plane of a point on $M$ with shortest distance to $o$ can not contain $o$. Consequently, with the help of Lemma \ref{lem3}, $o$ lies in the convex hull of $\d M$. Then Crofton's formula implies that $\Length[\d M]\geq 2\pi$ with equality if and only if $\d M$ is a great circle. The equality indeed holds as we pointed out above. So $\d M$ is a great circle. In particular $\Gamma=\d M$ is convex.

\section{Extending $M$ to a Closed Hypersurface $\ol M$}\label{sec:extension}
Having established the convexity of each boundary component $\Gamma$ of $M$, we will now extend $M$ to a closed $\C^1$ hypersurface $\ol M$  which is $\C^2$ except along some closed set $A$ of measure zero. Furthermore we will show that the image of the Gauss map  of $\ol M$ restricted to $A$ has measure zero as well. This involves gluing along each $\Gamma$ a suitable convex disk, which we construct with the aid of the following three lemmas. A \emph{convex cone}  is a closed convex proper subset $C$ of $\R^{n+1}$ such that for every $x\in C$, $\lambda x\in C$ for $\lambda\geq 0$. In particular $o\in C$. We say that $M\subset\R^{n+1}$ is a \emph{convex conical hypersurface} if it bounds a convex cone which has interior points.

\begin{lem}\label{lem:cone1}
Let $M\subset\R^{n+1}$ be a convex conical hypersurface which is $\C^{1}$  in the complement of $o$. Suppose that $M$ is not strictly convex at $o$. Then $M$ is a hyperplane.
\end{lem}
\begin{proof}
Since $M$ is not strictly convex at $o$, it must contain a line $L$ passing through $o$.
Let $H$ be a support hyperplane of $M$ at $o$. Then $H$ is tangent to $M\cap\S^n$ at the points $L\cap\S^n$. Consequently $H$ is the unique support hyperplane of $M$ at $o$. So if $C$ is the cone bounded by $M$, then the ``dual cone" $C^\circ$ of $C$, generated by all outward normals to support hyperplanes  of $C$ at $o$, consists of a single ray. Consequently the dual of the dual cone, $C^{\circ\circ}$ is a half-space. But $C^{\circ\circ}=C$, e.g., see \cite[p. 35]{schneider2014}. Thus $C$ is a half-space, which yields that $M$ is a hyperplane.
\end{proof}

\begin{lem}\label{lem:cone2}
Let $M\subset\R^{n+1}$ be a convex conical hypersurface which is $\C^{k\geq 2}$  in the complement of  $o$. Suppose that $M$ is strictly convex at $o$. Then for any ball $B$ centered at $o$ there exists a $\C^1$ convex hypersurface $\tilde M$ which coincides with $M$ outside $B$ and is $\C^k$ except along a pair of closed $\C^{k-1}$ hypersurfaces $\Gamma_i$, $i=1$, $2$. Furthermore, the images of $\Gamma_i$ under the Gauss map of $\tilde M$ have measure zero.
\end{lem}
\begin{proof}
After a homothety, we may assume that $B$ is the unit ball $B^{n+1}$. Further let us assume, after a rotation, that $(0,\dots,0,-1)$ lies in the interior of the convex region bounded by $M$, and $M$ intersects the $x_{n+1}=0$ hyperplane only at $o$. Let $H$ be the plane $x_{n+1}=c_0<0$ such that the component of $M\setminus H$ which contains the origin lies  in $B^{n+1}$, and $K$ be the convex set which lies below $H$ and inside $M$. Next, let $\tilde K$ be the union of all balls of radius $\delta$ contained inside $K$; see Figure \ref{fig:triangle}. We claim that if $\delta$ is sufficiently small, then $\tilde M:=\d \tilde K$ is the desired surface.

\begin{figure}[h]
\centering
\begin{overpic}[height=1.75in]{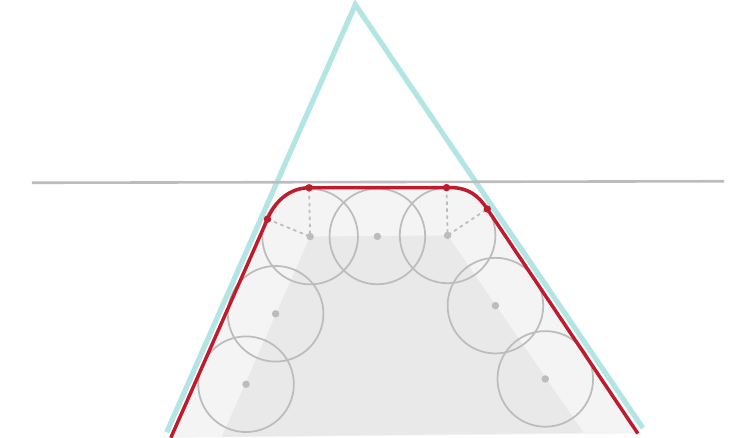}
\put(38,47){\Small $M$}
\put(40,35){\Small $\tilde M$}
\put(0,33){\Small $H$}
\put(49,28){\Small $\tilde K$}
\put(49,11){\Small $K_\delta$}
\end{overpic}
\caption{}\label{fig:triangle}
\end{figure}

First we check that $\tilde K$ is convex. Let $p_i\in \tilde K$, $i=1$, $2$. Then there are balls $B_i\subset K$ of radius $\delta$ which contain $p_i$. The convex hull, $\conv(B_1\cup B_2)$, of $B_1$ and $B_2$ lies in $K$,  since $K$ is convex. Note that $\conv(B_1\cup B_2)$ consists of all balls of radius $\delta$ centered at the segment connecting  centers of $B_i$. So $\conv(B_1\cup B_2)\subset\tilde K$, which completes the argument since $p_1p_2\subset \conv(B_1\cup B_2)$.

Second we check that $\d\tilde K$ is $\C^1$. Through every point of $\d \tilde K$  there passes a ball contained in $\tilde K$. Consequently the support hyperplane through every point of $\d \tilde K$ is unique. It follows that $\d \tilde K$ is $\C^1$  \cite[Thm. 1.5.15]{schneider2014}.

Third we check that $\tilde M$ coincides with $M$ outside $B^{n+1}$. To see this let  $M'$ denote the portion of $M$ outside the interior of a ball of radius $|c_0|/2$ centered at $o$, and $N$ be the inward normal of $M'$. Since $M'$ is $\C^2$, by the tubular neighborhood theorem, we may choose $\delta$ so small that the mapping $p\mapsto p+\delta N(p)$ is one-to-one on any given compact subset of $M'$, such as its boundary $\d M'$. Then, since $N$ is constant along each ray of $M'$, it follows that $p\mapsto p+\delta N(p)$ is one to one on $M'$. Hence through each point of $M'$ there passes a ball of radius $\delta$ which lies inside $M$. Further we can make sure that $\delta$ is smaller than the distance between $H$ and $M\cap \S^n$. Then all $\delta$-balls which intersect $M$ outside $B^{n+1}$ are inside $M$ and below $H$, and so they are contained in $K$. So $\tilde M$ coincides with $M$ outside $B^{n+1}$.

Fourth, we check the regularity of $\tilde M$. To this end note that the set $K_\delta\subset K$ which consists of the centers of all $\delta$-balls inside $K$ is itself a convex set. Choosing $\delta $ sufficiently small, we may assume that there exists a $\delta$-ball inside $K$ which is tangent to $H$ and disjoint from $M$. Then $X:=\d \tilde K \cap H$ is a convex body in $H$, comprised of the intersections of $\delta$-balls in $K$ with $H$. Accordingly, if we let $H_\delta$ be the hyperplane parallel to and below $H$ at the distance $\delta$ from $H$, then $X_\delta:=\d K_\delta\cap H_\delta$ is a convex body in $H_\delta$. So $\Gamma:=\d X_\delta$ is a hypersurface in $\d K_\delta$. Note that $\Gamma$ is the intersection of $H_\delta$ with the parallel hypersurface $M'_\delta$ of $M'$. Hence $\Gamma$ is $\C^k$, since $M'$ is $\C^k$, which  implies that $M'_\delta$ is $\C^k$. Indeed, at small distances, parallel hypersurfaces of $\C^k$ hypersurfaces are $\C^k$ \cite[Lem. 3.1.8]{ghomi:stconvex}.

Now $\Gamma$ determines two different hypersurfaces in $\tilde M$. One hypersurface, say $\Gamma_1$ is obtained by moving $\Gamma$ upward by a distance of $\delta$ along the normals to $H_\delta$. The other hypersurface, say $\Gamma_2$,  is obtained  by expanding $\Gamma$ along the outward normals to $M'_\delta$. Then $\Gamma_1$ and $\Gamma_2$ are both $\C^{k-1}$ hypersurfaces of $\tilde M$. Further these two hypersurfaces determine three regions in $\tilde M$: one, bounded by $\Gamma_1$, is just a flat disk in $X$; another, bounded by $\Gamma_2$, lies in $M$, while the third is an annular region, say $A$, bounded by $\Gamma_1$ and $\Gamma_2$. The first region is $\C^\infty$, since it lies in $H$, while the second region is $\C^k$ since it lies in $M$. It remains then to establish the regularity of the third region $A$.

 By construction, $A$ lies on the boundary of the set of all $\delta$-balls centered at $\Gamma$. Equivalently, $A$ lies on a tubular hypersurface of $\Gamma$ at the distance $\delta$. Since $\Gamma$ is $\C^k$, it follows that its tubular hypersurface is $\C^k$ as well, since the distance function of a $\C^k$ submanifold is $\C^k$ \cite[Sec. 2.4]{ghomi:stconvex}. Hence $A$ is $\C^k$ as desired.

 Finally we check that $\tilde N(\Gamma_i)$ has measure zero, where $\tilde N$ is the Gauss map of $\tilde M$. First note that $\tilde N(\Gamma_1)$ is a singleton, since $\tilde M$ is tangent to the hyperplane $H$ along $\Gamma_1$. Furthermore, $\tilde M$ is tangent to $M$ along $\Gamma_2$. Thus $\tilde N(\Gamma_2)=N(\Gamma_2)\subset N(M\setminus\{o\})$ where $N$ is the Gauss map of $M\setminus\{o\}$. But $N(M\setminus\{o\})$ has measure zero since $M$ is a convex cone. Indeed $N(M\setminus\{o\})=\partial N(M)$, and $N(M)$ is a convex subset of $\S^{n}$. So $\partial N(M)$ has measure zero, which in turn yields that $\tilde N(\Gamma_2)$ has measure zero, and completes the proof.
\end{proof}

\begin{lem}\label{lem:inversion}
Let $M\subset B^{n+1}$ be a compact convex hypersurface with free boundary on $\S^n$, and suppose that $o\not\in M$. Then the inversion of $M$ through $\S^n$ is again a compact convex hypersurface with free boundary on $\S^n$.
\end{lem}
\begin{proof}
For every point $x\in\R^{n+1}$, let $x':=x/\|x\|^2$ denote its inversion through $\S^n$.
Since $M$ is convex, through each point $p\in M$ there passes a support hyperplane $H$. Let $H^+$ be the side of $H$ where $M$ lies and $H^-$ be the opposite side.
We claim that $o\in H^-$ for all $p\in M$. Indeed, suppose towards a contradiction that $o\in \inte(H^+)$. Then, since $\d M$ also lies in $H^+$, $H^+$ contains the cone $C$ formed by connecting $o$ to points of $\d M$. But $C$ contains $M$, since $M$ is convex and $C$ is tangent to $M$ along $\partial M$, which ensures that every tangent hyperplane of $C$ is a tangent hyperplane of $M$, and therefore is a support hyperplane of $M$. Thus $H$ cannot be disjoint from $C$, and therefore is a support plane of $C$. In particular $H$ contains $o$, which is the desired contradiction. So we conclude that $o\in H^-$. Now for every $p\in M$,  either (i) $o\in H$, or (ii) $o\in\inte(H^-)$. In case (i), $H'=H$ and $(H^+)'=H^+$. Consequently $H$ is a support hyperplane of $M'$ at $p'$. In case (ii), $H'$ is a sphere passing through $o$, and $(H^+)'$ is the ball bounded by  $H'$. Thus $T_{p'} H'$ is a supporting hyperplane of $M'$ at $p'$. So $M'$ is convex since through each point of it there passes a support hyperplane. Further, $M'$ is orthogonal to $\S^n$ along $\d M'$ since inversion is a conformal map.
\end{proof}

Now for each component $\Gamma$ of $\d M$ we construct a convex disk $D_\Gamma$ as follows.
Let $C_\Gamma$ be the conical disk generated by connecting all points of $\Gamma$ to $o$. By Lemma \ref{lem:cone1}, $C_\Gamma$ is either a flat disk or else is strictly convex at $o$. In the latter case, let $\tilde C_\Gamma$ be the smoothing of $C_\Gamma$ near $o$ given by Lemma \ref{lem:cone2}, and $(\tilde C_\Gamma)'$ be the inversion of $\tilde C_\Gamma$ given by Lemma \ref{lem:inversion}.

First suppose that $M$ lies outside $\S^n$ near $\Gamma$.   If $\Gamma$ is not a great sphere,  i.e., the intersection of $\S^n$ with a hyperplane through $o$, set $D_\Gamma:=\tilde C_\Gamma$; otherwise, set $D_\Gamma:=C_\Gamma$. Next suppose that $M$ lies inside $\S^n$ near $\Gamma$. If $\Gamma$ is not a great sphere, set $D_\Gamma:=(\tilde C_\Gamma)'$; otherwise, we proceed as follows. Let $H$ be the hyperplane of $\Gamma$, $S$ be a sphere of radius $2$ centered at $o$, and $A$ be the annular region in $H$ bounded by $\Gamma$ and  $\Gamma':=S\cap H$. Take one of the hemispheres of $S$ bounded by $\Gamma'$, glue it to $A$, and smoothen the joint to obtain the desired  disk $D_\Gamma$ (the smoothing here is trivial, since we have a surface of revolution).

Now gluing each of the disks $D_\Gamma$ along the corresponding component $\Gamma$ of $\d M$ yields the desired closed surface $\ol M$. In particular  $\ol M$ is  $\C^1$, because it has  flat tangent cones at each point which vary continuously along $\ol M$, see \cite[Lem. 3.1]{ghomi&howard:tancones}.  Furthermore, $\ol M$ is $\C^2$ everywhere except possibly along each $\Gamma$, and  a pair of closed hypersurfaces $\Gamma_1$, $\Gamma_2$ in the interior of $D_\Gamma$ when $\Gamma$ is not a great sphere.
By Lemma \ref{lem:cone2}, $\ol N(\Gamma_i)$ has measure zero, where $\ol N$ is the Gauss map of $\ol M$. Finally note that $\ol N(\Gamma)$ has measure zero as well, since $\ol M$ is tangent to a convex cone along $\Gamma$, and as we argued at the end of the proof of Lemma \ref{lem:cone2}, the Gauss image of the lateral portion of a convex cone has measure zero.

\section{Convexity of $\ol M$}\label{sec:dCL}

To finish the proof  of Theorem \ref{thm:main} we need one last observation, which is essentially due to  do Carmo and Lima \cite{dCL69}. We mainly check that the stated $\C^\infty$ regularity in their arguments may be relaxed by a somewhat finer use of Morse theory. This yields the following generalization of Sacksteder's theorem in the compact case (which originally had required $\C^{n+1}$ regularity).

\begin{prop}\label{prop:C1}
Let $M$ be an immersed closed $\C^1$ hypersurface in $\R^{n+1}$. Suppose that $M$ is $\C^2$ and nonnegatively curved on $M\setminus A$, where $A$ is a closed subset of measure zero. Moreover, suppose that the image of $A$ under the Gauss map of $M$ has measure zero. Then $M$ is convex.
\end{prop}

The proof follows from the next four lemmas. For every $u\in\S^n$, let $h_u\colon M\to\R$ be the \emph{height function}
$
h_u(\cdot):=\langle \cdot, u\rangle.
$
 Note that  $h_u$ is $\C^1$ on $M$, and is $\C^2$ on $M\setminus A$.   The next lemma follows from Chern and Lashof \cite[Thm. 3]{chern&lashof:tight1} as indicated in \cite[Lem 3.2]{ghomi:shadow}.  Alternatively, one may apply a result of Kuiper \cite[Thm. 4]{Kui70} which applies  to topologically immersed hypersurfaces, together with Reeb's theorem  \cite[Thm. 4.1]{milnor:morse}. A \emph{critical point} of $h_u$ is a point where its gradient vanishes.

\begin{lem}[\cite{Kui70}]\label{lem:kuiper}
$M$ is convex if $h_u$ has only two critical points for almost every $u\in\S^n$.
\end{lem}
\begin{proof}
Let $C(h_u)$ be the set of critical points of $h_u$. Then $p\in C(h_u)$ if and only if $N(p)=\pm u$, where $N$ is the Gauss map of $M$. If $\# C(h_u)=2$ for almost all $u\in\S^n$, then, by the area formula \cite[Thm. 3.2.3]{federer:book}
$$
2 \vol(\S^n)=\int_{\S^n} \# C(h_u)\,du=\int_{\S^n}\# N^{-1}(\pm u)\,du=2\int_{M}|\det(dN_p)|\,dp=2\int_{M} |K|,
$$
where $K$ is the Gauss-Kronecker curvature of $M$. By assumption, $N$ is differentiable almost everywhere, and so the integrals above are well defined. Thus
$
\int_{M} |K|=\vol(\S^n),
$
or $M$ has ``minimal total absolute curvature", which yields that $M$ is convex by Chern and Lashof \cite[Thm. 3]{chern&lashof:tight1}.
\end{proof}

We say that a critical point $p$ of $h_u$ is   \emph{nondegenerate} provided that $p\in M\setminus A$ and  the eigenvalues of the Hessian of $h_u$ at $p$ are all nonzero. The following fact is stated for $\C^\infty$ hypersurfaces in  \cite[Lem. 2]{dCL69}. Here we apply the Morse inequalities to extend that result to the $\C^1$ case:

\begin{lem}[\cite{dCL69,kuiper72}]\label{lem:dCL}
If, for some $u\in\S^n$, all critical points of $h_u$ are nondegenerate  local extrema, then $h_u$ has only two critical points.
\end{lem}
\begin{proof}
It is clear from the proof of Morse inequalities \cite[Sec. 5]{milnor:morse} that they apply to any $\C^1$ function which is sufficiently smooth near its critical points, so that Morse's Lemma holds \cite[Lem. 2.2]{milnor:morse}. Kuiper \cite{kuiper72} proved that Morse's lemma holds for functions which are $\C^2$ near an isolated critical point; see also Ostrowski \cite{ostrowski68}. Thus Morse inequalities do indeed apply to $h_u$.
Let $C_\lambda$ be the number of critical points of $h_u$ of index $\lambda$, and $\beta_\lambda$ be the Betti numbers of $M$. By assumption $C_\lambda=0$ for $0<\lambda<n$.  So by \cite[Cor. 5.4]{milnor:morse} $C_0=\beta_0$ and $C_n=\beta_n$.
By Poincar\'{e} duality, $\beta_0=1=\beta_n$, which completes the proof.
\end{proof}

We say that $u\in\S^n$ is a \emph{regular value} of the Gauss map $N$ provided that $N^{-1}(u)\subset M\setminus A$, and $dN_p$ is nondegenerate for every $p\in N^{-1}(u)$.

\begin{lem}
If  $\pm u$ are regular values of $N$, then all critical points of $h_u$  are nondegenerate  local extrema.
\end{lem}
\begin{proof}
Let $p$ be a critical point of $h_u$. Then $N(p)=\pm u$. Thus $K(p)=\det (dN_p)\neq 0$, for $p\in M\setminus A$. Hence the principal curvatures $k_i$ of $M$ do not vanish at $p$. Further $k_i(p)k_j(p)\geq 0$, since these are the sectional curvatures of $M$  for $i\neq j$. Thus $k_i(p)$ all have the same sign. It remains only to recall the well-known fact that $k_i(p)$ are
the eigenvalues of the Hessian of $h_u$ at $p$, after we replace $N$ with $-N$ if necessary.
\end{proof}

To complete the proof of Proposition \ref{prop:C1} it only remains to observe that

\begin{lem}
For almost every $u\in\S^n$, $\pm u$ are regular values of $N$.
\end{lem}
\begin{proof}
Note that $\pm u$ are regular values of $N$, if $u$ is a regular value of $\pm N$. By Sard's theorem, the sets $\pm C$ of critical values of $\pm N$ on $M\setminus A$ have measure zero, since $N$ is $\C^1$ on $M\setminus A$. Furthermore $\pm N(A)$ have measure zero by assumption. So $X:=\pm N(A)\cup\pm C$ has measure zero and every $u$ in $\S^n\setminus X$ is a regular value of $\pm N$.
\end{proof}

\section{Constant $m^{th}$ Mean Curvature}\label{sec:cmc}
Here we prove Corollary \ref{cor:cmc}.
In 1958, Alexandrov  \cite{Ale62}  showed that any embedded closed hypersurface with constant mean curvature  in $\R^{n+1}$ is a round sphere, via his celebrated reflection method. In fact he established a more general result for  certain Weingarten hypersurfaces, e.g., see  \cite[Prop. 1.1]{Har78}. An immersed orientable hypersurface $M$ in $\R^{n+1}$ is  \emph{Weingarten}  if $W(k_1(p),\dots,k_n(p))$  is constant for some functional $W$ of its  principal curvatures $k_i$.  In particular $M$ has \emph{constant $m^{th}$ mean curvature}, for $1\leq m\leq n$, when $W$ is the  symmetric elementary polynomial
$$
\sigma_m(k_1,\dots, k_n)=\sum_{i_1<\dots<i_m}k_{i_1}\dots k_{i_m}.
$$
 Thus $m=1$, $2$, and $n$ correspond respectively to the mean, scalar, and   Gauss-Kronecker curvatures of $M$.
Hartman \cite{Har78} showed that a complete nonnegatively curved hypersurfaces with constant $m^{th}$ mean curvature is the product of a sphere and a Euclidean space. See also Rosenberg \cite{rosenberg:constant} for another proof, and Cheng and Yau \cite{cheng&yau} for more on the case $m=2$. On the other hand, for surfaces with boundary several fundamental problems in this area remain open. It is not known, for instance, if a compact embedded CMC surface in $\R^3$ with circular boundary is umbilical (i.e., a spherical cap or a flat disk) \cite{Lop13}.

\begin{proof}[Proof of Corollary \ref{cor:cmc}]
By Theorem \ref{thm:main}, $M$ is an embedded convex disk. Suppose first  that $\d M$ is a great sphere, i.e., it lies in a hyperplane $H$ passing through the origin. If $M$ lies in $H$ as well then we are done. Otherwise we may apply Alexandrov's reflection technique with respect to the hyperplanes orthogonal to $H$ to conclude that $M$ is ``axially symmetric" or ``rotational", as has been shown by Wente \cite[Thm. 1.1]{Wen80}, see also Koiso \cite{Koi86}. In particular note that the relevant (elliptic) maximum principles cited in \cite[p. 391--392]{Wen80} all apply to surfaces with constant $m^{th}$ mean curvature. We also recall that axial symmetry means that, after a rigid motion, $M$ is invariant with respect to the standard action of the orthogonal group $O(n)$ on $\R^{n}\times\{0\}\subset\R^{n+1}$, which fixes the $x_{n+1}$-axis.

Next we assume that $\d M$ lies in an open hemisphere. Then the reflection method  may be adapted  to this setting via rotating hyperplanes which pass through the origin. More precisely, suppose that  $\d M$ lies in the interior of the  upper hemisphere of $\S^n$. Then support hyperplanes $H$ of $M$ along $\d M$ intersect the hyperplane of the first $n$-coordinates along $(n-1)$-dimensional subspaces $L:=H\cap(\R^{n}\times \{0\})$. Instead of moving $H$ parallel to itself, we rotate it around $L$, which  is a well-known variation on Alexandrov's original technique, e.g., see \cite[p. 75]{Lop13}.  Once again it follows, as in \cite{Wen80}, that $M$ is symmetric with respect to a line passing through the origin, which after a rotation we may assume to be the $x_{n+1}$-axis.

Now it follows from the generalization of Delaunay's theorem by Hsiang  \cite{Hsi82,Hsi83}, see also Sterling \cite{Ste87}, that $M$ is a spherical cap or an equatorial disk. Indeed, other than spheres and minimal hypersurfaces,  all rotational hypersurfaces of constant $m^{th}$ mean curvature in $\R^{n+1}$ must be part of a periodic hypersurface. Since $M$ intersects its axis of symmetry, it cannot be extended to a rotational periodic hypersurface. Hence it must be either spherical or  minimal. Due to the free boundary condition, and the maximum principle, $M$ may be minimal only when $\d M$ is a great sphere, in which case $M$ is an equatorial disk. Otherwise $M$ will be a spherical cap which completes the proof.
\end{proof}

\section*{Acknowledgment}
We thank Stephanie Alexander, Ralph Howard, Rafael L\'opez, John Pardon, and Ivan Sterling for useful comments. Furthermore, we are indebted to the anonymous referee for extensive comments which led to a number of  corrections and improvements in the exposition of this work.

\bibliography{references}

\end{document}